\documentclass[letterpaper,11pt]{amsart}

\usepackage[left=1in,right=1in,top=.75in,bottom=.75in]{geometry}

\usepackage{latexsym,amssymb,amsmath,amsthm,amsopn,verbatim}

\newtheorem{theorem}{Theorem}[section]

\newtheorem{prop}[theorem]{Proposition}
\newtheorem{lemma}[theorem]{Lemma}

\theoremstyle{definition}
\newtheorem{definition}[theorem]{Definition}

\theoremstyle{definition}
\newtheorem{example}[theorem]{Example}

\theoremstyle{remark}
\newtheorem{remark}[theorem]{Remark}

\theoremstyle{definition}

\newtheorem*{theorem*}{Theorem \ref{maintheorem}}

\newcommand{\R}{\mathbb{R}}
\newcommand{\T}{\mathcal{T}}
\newcommand{\Z}{\mathbb{Z}}
\newcommand{\C}{\mathbb{C}}
\newcommand{\RP}{\R\mathbb{P}}
\renewcommand{\Re}{\textrm{Re}}

\renewcommand{\ker}{\textrm{ker}}
\newcommand{\im}{\textrm{im}}

\newcommand{\note}[1]{}
\begin{document}

\title{Folded Symplectic Toric Four-Manifolds}
\author{Christopher R. Lee}
\address{Department of Mathematics, University of Portland, Portland, OR 97203}
\email{leec@up.edu}
\date{\today}

\begin{abstract}
We show that two orientable, four-dimensional folded symplectic toric manifolds are isomorphic provided that their orbit spaces have trivial degree-two integral cohomology and there exists a diffeomorphism of the orbit spaces (as manifolds with corners) preserving orbital moment maps. 
\end{abstract} 

\maketitle
\section{Introduction}

The classification of compact, connected symplectic toric manifolds by their moment images was completed by Delzant in 1988 \cite{D}. Since then, a number of extensions of Delzant's Theorem have been proved. Notably, a classification of compact, symplectic toric {\em orbifolds} was given by Lerman and Tolman in \cite{LT} and, more recently, Karshon and Lerman classified {\em noncompact} symplectic toric manifolds (\cite{KL}). Generally speaking, classification results of this kind rely on existence and uniqueness results: for each classifying gadget one must show there exists a corresponding class of manifolds and that, up to some topological restriction, this class contains only one isomorphism type.  This paper deals with a different type of extension of Delzant's results. Instead of changing the general assumptions about the manifold, we instead focus on properties of the differential form. Namely, we consider {\em folded symplectic forms}. 

A folded symplectic form is, colloquially, a closed two-form that is symplectic away from a hypersurface, $Z$, in a $2n$-dimensional manifold, $M$, whose degeneracies are reasonably well-controlled on $Z$. Here is the precise definition.

\begin{definition}
Let $\omega$ be a two-form on $M^{2n}$ such that the top exterior power, $\omega^n$, is transverse to the zero section of the orientation bundle $\bigwedge^{n}(T^*M)$. The zero locus, $Z:=(\omega^n)^{-1}(0)$, is then a codimension one submanifold of $M$. Let $i:Z \hookrightarrow M$ be the inclusion. If $i^*\omega^{n-1}$ is nonvanishing on $Z$, the two-form $\omega$ is a {\bf folded symplectic form}. A {\bf folded symplectic manifold} is a pair $(M,\omega)$ where $\omega$ is a folded symplectic form on the manifold $M$. The hypersurface $Z$ is called the {\bf fold}. The complement, $M/Z$ of the fold is the {\bf symplectic locus}. A {\bf folded symplectomorphism} between folded symplectic manifolds $(M,\omega)$ and $(M',\omega')$ is a diffeomorphism $f:M \to M'$ such that 
\begin{displaymath}
f^*\omega' = \omega.
\end{displaymath}
\end{definition}

The theory of folded symplectic forms has some similarity with that of symplectic forms, but there are some important and striking differences. 

\begin{example}\label{foldedsympex}
\begin{enumerate}
\item The two-form
\begin{displaymath}
\omega_{\textrm{fstd}}=x_1dx_1\wedge{dy_1}+\sum_{j=2}^{n}dx_j \wedge dy_j
\end{displaymath}
is the standard folded symplectic form on real Euclidean space
\begin{displaymath}
\R^{2n} = \{(x_1,\ldots,x_n,y_1,\ldots,y_n) \hspace{1pt} | \hspace{1pt} x_j,y_j \in \R\}.
\end{displaymath}
The fold is the hypersurface $\{x_1=0\}$. By pulling $\omega_{\textrm{fstd}}$ back by a choice of isomorphism,  any even dimensional vector space admits a folded symplectic form. In particular, the standard folded symplectic form on $\C^n$ is
\begin{displaymath}
\frac{i}{2}\Re(z_1)dz_1\wedge d\bar{z}_1 + \frac{i}{2}\sum_{j=2}^{n} dz_j \wedge d\bar{z}_j.
\end{displaymath}

\item Let $\Sigma$ be a closed, orientable surface and $\Omega$ a nonvanishing two-form on $\Sigma$. Let $f:\Sigma \to \R$ be a smooth function transverse to $0 \in \R$. Then, $f\Omega$ is a folded symplectic form on $\Sigma$.  For example, the two-form $\omega=hdh\wedge{d\theta}$ is a folded symplectic form on the sphere $S^2$ with fold $\{h=0\}$. 

The antipodal map
\begin{eqnarray*}
f:S^2 &\to& S^2 \\
(h,\theta) & \mapsto & (-h,\theta+\pi)
\end{eqnarray*}
satisfies $f^*\omega = \omega$ and so $\omega$ descends to give the quotient  $S^2/f$ the structure of a folded symplectic manifold. This quotient is diffeomorphic to the real projective plane $\RP^2$ and so, in contrast to symplectic manifolds, folded symplectic manifolds need not be orientable. 

\item Cannas da Silva, Guillemin, and Woodward show (in \cite{dSGW}), that every even-dimensional sphere, $S^{2n}$, has a folded symplectic form obtained by pulling back the standard symplectic form on $\R^{2n}$ by the map on $S^{2n}$ that folds about a great circle. Note that a sphere admits a symplectic form if and only if it is two-dimensional.

\item In \cite{CdS2}, Cannas da Silva proves that every orientable four-manifold admits a folded symplectic structure. This result indicates that folded structures are natural, ubiquitous, and worthy of further study.

\item If $(M,\omega)$ is a folded symplectic manifold  and $(N,\eta)$ is a symplectic manifold, then $M \times N$ is a folded symplectic manifold with fold-form $\omega \oplus \eta$. Note, however, that the product of two arbitrary folded symplectic manifolds may not be a folded symplectic manifold: the form $h_1dh_1\wedge d\theta_1 + h_2dh_2\wedge d\theta_2$ is {\em not} a folded symplectic form on $S^2 \times S^2$ since the second wedge power is not transverse to the zero section of $\bigwedge^{2n}(T^*(S^2 \times S^2))$. 
\end{enumerate}
\end{example}

By Darboux's Theorem, we know that every symplectic manifold looks locally like $\R^{2n}$ with its standard symplectic form. Similarly, there is a local normal form theorem for folded symplectic manifolds. Following \cite{dSGW}, we assume that $M$ is an {\em oriented} folded symplectic manifold and define a vector bundle $\ker(\omega) \to Z$ by setting
$$
\ker(\omega)_{z} = \{X \in T_{z}M \hspace{2pt} | \hspace{2pt} \iota_{X}\omega_{z} = 0 \}
$$
for each $z \in Z$. This is a rank two bundle and is oriented over $Z$ via the orientations induced on $i^*TM/\ker(\omega)$ by $\omega$ and that of $TM$. Using the inclusion $i:Z \to M$, we regard $TZ$ as the subbundle of $TM$ of vectors tangent along $Z$. Then, the intersection $\ker(\omega) \cap TZ \to Z$ is an oriented line bundle of tangent vectors along $Z$. We also note that
$$
T_{z}Z+\ker(\omega)_{z} = T_{z}M
$$
since the kernel of $\omega_{z}$ is two-dimensional and $\omega^{n-1}$ is nonvanishing on $Z$. The following it our version of the statement of Theorem 1 in \cite{dSGW}.

\begin{theorem}\label{thmoneindSGW}
Let $(M,\omega)$ be a compact, oriented, folded symplectic manifold with fold $Z$. Choose a positively oriented section $V$ of $\ker(\omega)\cap TZ$. Then, for any $\alpha \in \Omega^{1}(Z)$ with $\alpha(V) \equiv{1}$, there exists a neighborhood $U$ of $Z$ and a diffeomorphism $\psi:Z\times{\R} \to U$ such that $\psi^*\omega = p^*i^*\omega+d(t^2p^*\alpha)$ where $p:Z\times{\R} \to Z$ is projection on the first factor and $i:Z \to Z\times{\R}$ is the inclusion. 
\end{theorem}

For the remainder of this paper, our focus will be on Hamiltonian actions of tori on folded symplectic manifolds. 

\begin{definition}
An action of a Lie group $G$ on a folded symplectic manifold $(M,\omega)$ is a {\bf folded Hamiltonian action} if there is a {\bf moment map}. That is, for every $\xi$ in the Lie algebra $\mathfrak{g}$ of $G$ there exists a smooth function $\Phi: M \to \mathfrak{g}^*$ that is equivariant with respect to the coadjoint action of $G$ on $\mathfrak{g}^*$ and is such that $\iota(\xi_{_M})\omega = -d\Phi^{\xi}$. (Here, $\xi_{M}$ is the vector field on $M$ induced by $\xi$, {\em i.e.}, $\xi_M = \left.\frac{d}{dt}\right|_{t=0} \exp{(t\xi)}$.) A folded symplectic manifold admitting a folded Hamiltonian action of a Lie group $G$ will be called a {\bf folded Hamiltonian $G$-manifold} and we denote such via a triple $(M,\omega,\Phi)$. An {\bf isomorphism} of folded Hamiltonian $G$-manifolds $(M,\omega,\Phi)$ and $(M',\omega',\Phi')$ is an equivariant folded symplectomorphism $\psi:M \to M'$ such that $\psi^*\Phi' = \Phi$.   For a torus $\T$, a {\bf folded symplectic toric manifold} is a compact, connected, folded Hamiltonian $\T$-manifold $(M,\omega,\Phi)$ such that the torus action is effective and $2\dim{\T}=\dim{M}$. 
\end{definition}

\begin{remark}
The assumption that a folded symplectic toric manifold be connected is included for convenience. However, unless otherwise specified, we will assume that a folded symplectic toric manifold is compact.  
\end{remark}

\begin{comment}
\begin{remark}
Note that if a Lie group $G$ acts on a folded symplectic manifold $(M,\omega)$ preserving the folded symplectic form, the action also preserves the fold $Z$. 
\end{remark}
\end{comment}

Since the torus is Abelian, the coadjoint action of $\T$ on the dual Lie algebra $\mathfrak{t}^*$ is trivial. Hence, the moment map $\Phi:M \to \mathfrak{t}^*$ of a folded symplectic toric manifold is invariant and thus descends to a map
\begin{displaymath}
\underline{\Phi}:M/\T \to \mathfrak{t}^*
\end{displaymath}
called the {\bf orbital moment map}. The orbital moment map has the same image in $\mathfrak{t}^*$ as the moment map; that is, $\Phi = \underline{\Phi} \circ \pi$ where $\pi:M \to M/\T$ is the orbit projection. 

A general classification of folded symplectic toric manifolds is not yet complete. A certain subclass, however, has been classified. In \cite{CdSGP}, Cannas da Silva, Guillemin, and Pires classify the so-called {\em origami manifolds}, which are orientable folded symplectic toric manifolds having extra conditions regarding the restriction of the fold form to the folding hypersurface. In what follows, we present a first step in the classification of folded symplectic toric manifolds, a uniqueness result in the case of compact, orientable, four-dimensional manifolds. Our main theorem follows.

\begin{theorem}\label{maintheorem}
Let $(M, \omega, \Phi)$ and $(M', \omega', \Phi')$ be orientable folded symplectic toric four-manifolds. Suppose that $H^2(M/\T; \Z)=0$. Then, $M$ and $M'$ are isomorphic if there exists a diffeomorphism $h:M/\T \to M'/\T$ of manifolds with corners such that $h^*\underline{\Phi}' = \underline{\Phi}$. 
\end{theorem}

We will prove this theorem in two steps. The first establishes a condition under which folded symplectic toric manifolds are locally isomorphic. We then turn to questions of global isomorphisms. In the last section, we provide a proof of Theorem \ref{maintheorem} by passing our local results to the global setting.

\section{The Local Picture}\label{localsection} In this section, we establish the existence of local equivalence of folded symplectic toric four-manifolds given a diffeomorphism of orbit spaces that preserves the orbital moment maps.  The first step is understanding the differential structure of the orbit space. We then use that structure to help in determining when two folded symplectic toric four-manifolds are locally isomorphic, as defined below. 

\begin{definition}
Let $(M,\omega,\Phi)$ and $(M',\omega',\Phi')$ be folded symplectic toric manifolds. Denote by $\pi:M \to M/\T$ and $\pi':M' \to M'/\T$ the orbit maps. We say that $(M,\omega,\Phi)$ and $(M',\omega',\Phi')$ are {\bf locally isomorphic} if there exists a diffeomorphism $h:M/\T \to M'/\T$ such that for each $x \in M/\T$, there exists a neighborhood $U$ of $x$ and an isomorphism
$$
\psi_{x}:\pi^{-1}(U) \to (\pi')^{-1}(h(U))
$$
of (open) folded symplectic toric manifolds such that $(h \circ \pi)(p) = (\pi' \circ \psi_{x})(p)$ for every $p \in \pi^{-1}(U)$.  
\end{definition}

\begin{remark}\label{orbitpreservingismomentpreserving}
When two folded symplectic toric manifolds $(M,\omega,\Phi)$ and $(M',\omega',\Phi')$ are locally isomorphic, it can be convenient to identify their orbit spaces via a choice of diffeomorphism that lifts locally to isomorphisms. When this is done, we can write $\phi' = \underline{\Phi} \circ \pi'$ where $\pi':M' \to M/\T$ is the orbit projection and $\underline{\Phi}:M/\T \to \mathfrak{t}^*$ is the orbital moment map.  
\end{remark}

\subsection{The structure of the orbit space}  We certainly cannot expect the orbit space of a folded symplectic toric four manifold to be a manifold itself, but there is a differential structure on the orbit space that we may exploit. In particular, we will show that the orbit space is a {\bf manifold with corners}. Briefly, a manifold with corners is a second countable Hausdorff space with a maximal atlas of charts modeled on sectors $[0,\infty)^k \times \R^{n-k}$ with $0 \leq k \leq n$. The reader is referred to Appendix A of \cite{KL} for more details.  To establish such a differential structure on the orbit space, we proceed as follows. At points not in the fold, we may appeal to well-known results from symplectic geometry. At points in the fold, however, it is necessary to know the types of stabilizer subgroups that can appear. We begin by observing that, in the fold, orbits may be stabilized by a circle subgroup of the torus. 

\begin{example}\label{circlestabilizers}
Consider the product $S^2 \times S^2$, equipped with local coordinates $(h_1,\theta_1,h_2,\theta_2)$, $-1 \leq h_j \leq 1$.  This is a folded symplectic manifold with folded symplectic form $\omega = h_1dh_1 \wedge d\theta_1 + dh_2\wedge d\theta_2$. The fold is the copy of $S^1 \times S^2$ defined by $\{h_1=0\}$.  The folded symplectic form is invariant under the action of the two-torus given by
$$
(\lambda_1,\lambda_2)\cdot(h_1,\theta_1,h_2,\theta_2) = (h_1,\theta_1+\lambda_1,h_2,\theta_2+\lambda_2).
$$
In fact, this is a folded Hamiltonian action with moment map
$$
\Phi(h_1,\theta_1,h_2,\theta_2) = \left(\frac{1}{2}h_1^2,h_2\right).
$$
The circle $S^1$ stabilizes every point in $\Phi^{-1}(0,1)$.  
\end{example}

The other stabilizers that can appear in the fold are limited. In particular, the following result states that, at least in dimension four, no nontrivial discrete subgroup of the torus can stabilize points in the fold. This in turn implies that for a four-dimensional folded toric symplectic manifold, the set of free orbits is dense (Theorem 4.27 in \cite{Ka}). 

\begin{lemma}\label{nofixedpoints}\label{principalorbitsarefree}
Suppose a torus, $\T^n$, acts effectively on a connected, orientable manifold, $M^{2n}$, preserving a hypersurface $Z$. Then, there are no fixed points of the action in $Z$. Furthermore, if $M$ is four-dimensional, the only discrete subgroup of $\T^n$ stabilizing points in the fold is the trivial subgroup.
\end{lemma}

\begin{proof}
By the Slice Theorem, there is a neighborhood $U$ of an orbit in $M$ that is equivariantly diffeomorphic to $\T \times_{G_p} \mathcal{V}_{p}$ where $p$ is any point in the orbit, $\mathcal{V}_p$ is the differential slice at $p$, and $G_p$ is the stabilizer of the point $p$. Since the torus acts effectively on $M$, it acts effectively on $U$, hence on $\T\times_{G_p}\mathcal{V}_p$. Since the torus is Abelian, the isotropy action of $G_p$ on $\mathcal{V}_p$ is effective. If $p$ is a fixed point, then $G_p=\T$. So, if $p\in Z$, by utilizing a choice of invariant metric, we get a faithful representation
$$
\T^n \hookrightarrow O(2n-1) \times O(1)
$$
since $Z$ is an invariant hypersurface. The torus is connected, and so the image lies in $SO(2n-1)$. Since the maximal torus in $SO(2n-1)$ has dimension $n-1$, we have arrived at a contradiction. Hence $p$ cannot be a fixed point in $Z$. 

Now suppose that $\dim{M}=4$ and let $H$ be a  nontrivial discrete subgroup of $\T$ stabilizing a point $p \in Z$. Again by the Slice Theorem, there is a neighborhood of an orbit through $p$ is equivariantly diffeomorphic to $U=\T \times_{H} (\R \times \R)$. Since $M$ is orientable, so are $U$ and $Z$. We then have a faithful representation $H \hookrightarrow SO(1) \times SO(1)$ and conclude that $H$ is trivial. \end{proof}

We are now in a position to prove the following. 

\begin{prop}\label{orbitspaceisamanifoldwithcorners}
Let $(M,\omega,\Phi)$ be an orientable, folded symplectic toric four-manifold. Then, the orbit space $M/\T$ is a smooth surface with corners, wherein the fold $Z$ descends to a one-dimensional submanifold with boundary. 
\end{prop}

\begin{proof}
On the components of the symplectic locus $M\setminus{Z}$ the result follows by the local normal form theorem for symplectic toric manifolds (Proposition 1.1 in \cite{KL}). On the fold $Z$, Lemma \ref{nofixedpoints} and Example \ref{circlestabilizers} assert that the only possible stabilizers are $S^1$ and the trivial subgroup. A neighborhood of a point in the orbit space with trivial stabilizer is homeomorphic to an open two-disk. If the orbit corresponding to a point in $Z/T$ has a circle stabilizer, then it has a neighborhood homeomorphic to $\R^2/S^1\times \R = [0,\infty) \times \R$. 
\end{proof}

\subsection{Local equivalence of folded symplectic toric manifolds}
Local models for neighborhoods of orbits through points in a folded symplectic toric manifold $(M,\omega,\Phi)$ depend necessarily on whether the point lies in the symplectic locus $M\setminus{Z}$ or in the fold $Z$. Knowing the image of the moment map is not enough to specify the local structure of an arbitrary folded symplectic toric manifold; we must also incorporate information about the orbital moment map. Critical to this is the fact that orbits are isotropic. Recall that a submanifold $L$ of a manifold $N$ equipped with a two-form $\Omega$ is {\bf isotropic} if $\Omega|_{L}$ vanishes. 

\begin{prop}\label{orbitsareisotropic}
Orbits in a folded Hamiltonian $\T$-manifold $(M,\omega,\Phi)$ are isotropic. 
\end{prop}
\begin{proof}
Since $\T$ is Abelian, the moment map $\Phi:M \to \mathfrak{t}^*$ is invariant and hence constant on orbits. Therefore, for any $\xi,\zeta \in \mathfrak{t}$ and $p \in M$,  
\begin{eqnarray*}
0 &=& \iota(\xi_{_M}(p))d\Phi_{p}^{\zeta} \\
&=& \omega_{p}(\xi_{_M}(p),\zeta_{_M}(p)).
\end{eqnarray*}\end{proof}

\begin{lemma}\label{Vtangenttoorbits}
Let $(M,\omega,\Phi)$ be an orientable, folded symplectic toric four-manifold. If $V$ is a section of $\ker(\omega) \cap TZ$, then $V$ is tangent to orbits in $Z$.
\end{lemma}
\begin{proof}
We first consider a neighborhood of a free orbit $\T \times \R$ in $Z$ with coordinates $(\theta_1,\theta_2,s,0)$. By Proposition \ref{orbitsareisotropic}, orbits are isotropic so the restriction of the folded symplectic form to this orbit is
\begin{eqnarray*}
i^*\omega &=& f_1d\theta_1\wedge{ds}+f_2d\theta_2\wedge{ds} \\ 
&=& (f_1d\theta_1+f_2d\theta_2)\wedge{ds}
\end{eqnarray*}
where $f_1$ and $f_2$ are invariant functions. It follows that if $V$ is in the kernel of $i^*\omega$, then  $V$ is tangent to the orbit. 

Now we consider singular orbits. By Proposition \ref{orbitspaceisamanifoldwithcorners}, the orbit space $Z/\T$ is a compact one-manifold whose boundary components are images of singular orbits under the orbit map. Since the boundary of a compact one-manifold consists of isolated points and the orbit map is continuous,  singular orbits are isolated in $Z$. Pick a singular orbit $\mathcal{O}$ in $Z$ and suppose $V$ is a section of $\ker(\omega) \cap TZ$ that is not tangent to $\mathcal{O}$. Since the orbit is isolated, there exists a neighborhood $U$ of $\mathcal{O}$ whose intersection with the set of singular orbits contains only $\mathcal{O}$. Let $\phi_{t}$ denote the flow of $V$. Then, since $V$ was assumed not to be tangent to $\mathcal{O}$ and free orbits are dense by Lemma \ref{principalorbitsarefree},  if $x \in \mathcal{O}$ there exists $\epsilon >0$ such that  $\phi_{\epsilon}(x) \not\in \mathcal{O}$ and $\phi_{\epsilon}(x)$ lies in a free orbit $\mathcal{F}$. Since $V$ is tangent to free orbits and $\T^2$ is compact, $\phi_{t}(x) \in \mathcal{F}$ for all $t \in \R$. This contradicts the fact that $\phi_{-\epsilon} \in \mathcal{O}$. Hence, $V$ is tangent to the singular orbit $\mathcal{O}$.   \end{proof}

In order to get local isomorphisms of folded symplectic toric manifolds, we need a diffeomorphism of orbit spaces that lifts locally to isomorphisms. We now show that any diffeomorphism of orbit spaces that preserves the orbital moment maps also preserves orbital folds.  

\begin{lemma}\label{orbitdiffeospreserveZ}
Let $(M,\omega,\Phi)$ and $(M',\omega',\Phi')$ be orientable, folded symplectic toric four-manifolds. Suppose $h:M/\T \to M'/\T$ is a diffeomorphism such that $h^*\underline{\Phi}' = \underline{\Phi}$. Then, $h(Z/\T) = Z'/\T$. 
\end{lemma}
\begin{proof}
Choose $p \in \pi^{-1}(x)$ and consider
\begin{eqnarray*}
F:\mathfrak{t} &\to& T^*_pM \\
\xi &\mapsto& \iota(\xi_{_M})\omega_{p}.
\end{eqnarray*}
This map is the transpose of 
$$
d\phi_{p}:T_pM \to\mathfrak{t}^*
$$
since $\langle d\Phi(Y),\xi \rangle = \omega(\xi_{_M},Y)$ for any vector field $Y$ on $M$. It follows that the image of $d\Phi_p$ is isomorphic to the annihilator in $\mathfrak{t}^* $ of the kernel of $F$. The kernel of $F$ consists of vectors in the stabilizer algebra $\mathfrak{t}_{p}$ as well as any $\xi \in \mathfrak{g}$ with $\xi_{_M}(p) \in \ker(\omega)_p$. By Lemma \ref{Vtangenttoorbits}, any nonzero vector spanning $\ker(\omega)_p \cap T_pZ$ lies in the image of the map $\mathfrak{g} \to T_pM$ given by $\xi \mapsto \xi_{_M}(p)$.  Hence, $p$ is a fold point if and only if $\dim(\im(d\Phi_p)) = \dim(\mathfrak{t}^{\circ}_p)-1$ where $\mathfrak{t}^{\circ}_{p}$ is the annihilator of $\mathfrak{t}_p$. 

Let $x \in M/T$. Denote by $\textrm{stab}(x)$ the stabilizer of any point $p \in \pi^{-1}(x)$ where $\pi:M \to M/\T$ is the orbit projection. Since $h:M/\T \to M'/\T$ is a diffeomorphism, the proof of Proposition \ref{orbitspaceisamanifoldwithcorners} implies that $\textrm{stab}(x) = \textrm{stab}(h(x))$. This in turn implies that $\mathfrak{t}^{\circ}_{p} = \mathfrak{t}^{\circ}_{p'}$ for any $p' \in (\pi')^{-1}(h(x))$. The result then follows since $\im(d\Phi) = \im(d\underline{\Phi})$.  \end{proof}

The next result states that moment preserving diffeomorphisms also preserve restriction of the folded symplectic form to the fold. 

\begin{lemma}\label{lemmarestriction}
Let $(M,\omega,\Phi)$ and $(M',\omega',\Phi')$ be orientable, folded symplectic toric four-manifolds. Suppose $g:M \to M'$ is an equivariant diffeomorphism such that $g^*\Phi' = \Phi$. Then, $i^*\omega = i^*(g^*\omega')$ where $i:Z \hookrightarrow M$ is the inclusion of the fold.
\end{lemma}
\begin{proof}
It suffices to show that $i^*\omega = i^*(g^*\omega')$ on a dense set and hence, by Lemma \ref{principalorbitsarefree},  it is enough to show the condition holds on free orbits. Let $U'$ be a neighborhood in $M'$ of a free orbit in $Z'$. By the Slice Theorem, $U$ is equivariantly diffeomorphic to $\T^2 \times \R \times \R$. Choose coordinates $(\theta'_{1}, \theta'_{2}, s',t')$ on $U$ where fold points have $t'=0$. Let $U = g^*(U')$ be given coordinates $(\theta_1,\theta_2,s,t) = (g^*\theta'_{1}, g^*\theta'_{2}, g^*s',g^*t')$. With respect to theses coordinates, using the fact that the orbits are isotropic, the folded symplectic forms on $U$ and $U'$ are
$$
\omega = f_1ds \wedge{d\theta_1} + f_2ds \wedge{d\theta_2} + \alpha \wedge dt $$
and
$$\omega' = f'_1ds' \wedge{d\theta'_1} + f'_2ds' \wedge{d\theta'_2} + \alpha' \wedge dt' 
$$
where the coefficient functions are invariant and $\alpha,\alpha'$ are invariant one-forms on $\T^2 \times \R \times \R$. It follows that 
$$
d\Phi = (f_1ds + k_1dt, f_2ds+k_2dt)
$$
and
$$d\Phi' = (f'_1ds' + k'_1dt', f'_2ds'+k'_2dt').
$$
Now, since $g^*\Phi' = \Phi$, and $d(g^*\Phi') = g^*(d\Phi'$), we see that $f_1 = g^*f'_1$ and $f_2 = g^*f'_2$. Thus, restricting to the fold $Z$, we have $i^*\omega = i^*(g^*\omega)$.  \end{proof}

\begin{remark}\label{locallemmarestriction}
If $U$ is an invariant open set in $M$ and $g:U \to U'$ is an equivariant diffeomorphism preserving moment maps, then the conclusion of Lemma \ref{lemmarestriction} still holds. If $i:U \cap Z \to U$ is the (restriction of) the inclusion of the fold, then $i^*\omega = i^*(g^*\omega')$. 
\end{remark}

We now state and prove an equivariant version of Theorem 1 in \cite{dSGW}, which provides a local normal form for a folded symplectic form on an orientable manifold. 

\begin{lemma}\label{equivariantlocalnormalform}
Let $G$ be a compact Lie group acting on a folded symplectic manifold $M$ preserving folded symplectic forms $\omega_0$ and $\omega_1$. Let $i:Z \hookrightarrow M$ be the inclusion of the fold. Then, if $i^*\omega_0=i^*\omega_1$, there exists an invariant neighborhood $U$ of $Z$ and an equivariant diffeomorphism $\psi:U \to U$ such that $\psi^*\omega_1 = \omega_0$. 
\end{lemma}
\begin{proof}
For each $z \in Z$, choose a basis $\{V_{z},W_{z}\}$ of $\ker(\omega)_z$ with $W_{z} \in T_{z}M$ and $V_{z} \in T_{z}Z$. We can equivariantly identify an invariant neighborhood of $Z$ with $Z \times \R$ in such a way that the vector field $W$ corresponds to $\frac{d}{dt}$. By the proof of Theorem 1 in \cite{dSGW}, we may write $\omega_0 = p^*i^*\omega_0+t\mu_0$ where $p:Z \times \R \to Z$ is the projection and $\mu_0(W_z,V_z)_z >0$ for all $z \in Z$. Let $\eta = p^*i^*\omega_0 + d(t^2p^*\alpha)$ where $\alpha$ is an invariant one-form such that $\alpha(V)>0$. Then, $\eta$ is an invariant folded symplectic form on $Z \times \R$ with fold $Z$. 

For all $s \in [0,1]$, the form $\omega_{s}:=(1-s)\omega_0+s\eta$ is an invariant folded symplectic form with fold $Z$. Since $i^*\omega_0=i^*\eta$ and $\omega_0-\eta$ is closed, we may find an invariant one-form $\beta$ on $Z \times \R$ such that $\beta_z=0$ for all $z \in Z$ and $d\beta = \omega_0-\eta$. For each $s \in [0,1]$ there is an invariant vector field $X_s$ such that $\iota(X_s)\omega_s = \beta$. The flow $\{\phi_s\}$ of $X_s$ consists of equivariant diffeomorphisms and $\phi_{1}^*\eta = \omega_0$. Since $\beta_{z}=0$ for all $z \in Z$, the flow of $X_s$ fixes $Z$ for all $s \in [0,1]$. 

We next note that
$$
\begin{array}{rcl}
\mathcal{L}_{X_{s}}\omega_{s} &=& \iota(X_s)d\omega_{s}+d\iota(X_s)\omega_s \\
&=& d\beta
\end{array}
$$
so that 

$$
\begin{array}{rcl}
\frac{d}{ds}(\phi_{s}^*\omega_{s}) &=& \phi^*_{s}\left(\frac{d}{ds}\omega_{s}+\mathcal{L}_{X_{s}}\omega_{s}\right) \\
&=& \phi^*_{s}(\eta-\omega+d\beta) \\
&=& \phi^*_{s}(\eta-\omega+\omega-\eta) \\
&=& 0.
\end{array}
$$
We conclude that $\phi^*_{s}\omega_{s}$ is independent of $s$. Noting that $\omega=\phi^*_{0}\omega_{0}=\phi^*_{1}\omega_1=\phi^*_{1}\eta$, we define $\psi_0 := \phi_1$. This yields an equivariant diffeomorphism satisfying $\psi_0^*\omega_0 = \eta$.

Since $i^*\omega_0 = i^*\omega_1$, this construction works for $\omega_1$ as well so that there is an equivariant diffeomorphism $\psi_1$ satisfying $\psi_1^*\omega_1 = \eta$. The equivariant diffeomorphism $\psi:=\psi_{1} \circ \psi_0^{-1}$ then satisfies $\psi^*\omega_1 = \omega_0$. \end{proof}

\begin{remark}\label{localequivariantlocalnormalform}
From the proof of Lemma \ref{equivariantlocalnormalform}, we can see that if $X$ is a closed, invariant subset of the fold $Z$ (in particular, an orbit), and there is an invariant neighborhood of $X$ such that the restrictions of folded symplectic forms $\omega_0$ and $\omega_1$ to $X$ are equal, then we can find an equivariant diffeomorphism $\psi$ of a neighborhood $U$ of $X$ in $M$ such that $\psi^*\omega_1=\omega_0$.
\end{remark}

\begin{theorem}\label{localstructure}
Let $(M, \omega, \Phi)$ and $(M', \omega', \Phi')$ be orientable, compact, folded symplectic toric four-manifolds. If there exists a diffeomorphism $h:M/\T \to M'/\T$ of manifolds with corners such that $h^*\underline{\Phi}' = \underline{\Phi}$, then $M$ and $M'$ are locally isomorphic. 
\end{theorem}

\begin{proof}
The result follows from Lemma B.4 in \cite{KL} on the symplectic loci $M\setminus{Z}$ and $M'\setminus{Z'}$. We therefore need only focus on points in the fold. Choose $x \in Z/\T$. By the Slice Theorem and Lemma \ref{orbitdiffeospreserveZ}, the diffeomorphism $h \colon M/\T \to M'/\T$ lifts to an equivariant diffeomorphism $g \colon N \to N'$ of neighborhoods of the orbits $\pi^{-1}(x)\subset Z$ and $(\pi')^{-1}(h(x)) \subset Z'$ respectively. Since $h$ preserves orbital moment maps, $g$ preserves the restrictions of $\Phi$ and $\Phi'$ to $N$ and $N'$. By Lemma \ref{lemmarestriction} and Remark \ref{locallemmarestriction}, $i^*\omega = i^*(g^*\omega)$ on $N$ where $i\colon N \cap Z \to N$ is the inclusion. Lemma \ref{equivariantlocalnormalform} and Remark \ref{localequivariantlocalnormalform} now imply that there exists a neighborhood $U$ in $N$ containing $\pi^{-1}(x)$ and an equivariant diffeomorphism $\psi\colon U \to U$ with $\psi^*(g^*\omega') = \omega$. It follows that $\phi=g \circ \psi$ is an equivariant diffeomorphism on a neighborhood of the orbit $\pi^{-1}(x)$ such that $\phi^*\omega' = \omega$. \end{proof}

\section{The Global Picture} 

As we saw in the previous section, a diffeomorphism of orbit spaces preserving orbital moment maps is enough to guarantee that two folded symplectic toric four manifolds are locally isomorphic. We will now show that under a suitable topological restriction on the orbit spaces, the existence of a moment-preserving diffeomorphism of folded symplectic toric four manifolds gives rise to an isomorphism between them. In all that follows, when given a diffeomorphism $g:M \to M'$ of folded symplectic toric four-manifolds, we denote by $\omega_{s}$ the convex combination of $\omega$ and $g^*\omega'$. That is, for $s \in [0,1]$, $\omega_s = (1-s)\omega+sg^*\omega'$.

\begin{lemma}\label{kerneltalk}
Let $(M,\omega,\Phi)$ and $(M'\omega',\Phi')$ be orientable, folded symplectic four-manifolds, $Z \subset M$ the folding hyperurface, and $g:M \to M'$ an equivariant diffeomorphism such that $g^*\Phi' = \Phi$. Then
\begin{itemize}
\item[(i)] $\ker(\omega) \cap TZ = \ker(\omega_{s}) \cap TZ$, and
\item[(ii)] for each $z \in Z$, if $W_{z} \in \ker(\omega)_{z} \cap T_{z}M$ and $W^s_{z} \in \ker(\omega_s)_{z} \cap T_{z}M$, $W^s_z$ and $W_z$ differ by a vector tangent to the orbit through $z$.
\end{itemize}
\end{lemma}
\begin{proof}
By Lemma \ref{lemmarestriction}, if $i:Z \hookrightarrow M$ is the inclusion, then $i^*\omega = i^*g^*\omega$ everywhere on $Z$. Hence, $i^*\omega_{s} = i^*\omega$. This establishes (i). 

As in the proof of Lemma \ref{Vtangenttoorbits}, it suffices to show (ii) assuming the orbit through $z$ is free. On a neighborhood $U$ of such an orbit, the projection $\pi:U \to U/\T$ is a map of manifolds and
\begin{eqnarray*}
d\pi(\ker(g^*\omega)) &=& d\pi(d(g^{-1})(\ker{\omega})) \\
&=& d\pi(\ker(\omega)),
\end{eqnarray*}
which proves (ii) on free orbits. \end{proof}

Before continuing, it will be beneficial to recall some facts about the {\bf basic cohomology} of a manifold $M$ admitting an action of a compact Lie group $G$. A {\bf basic form} on $M$ is an element $\beta \in \Omega^*(M)$ such that $\beta$ is invariant and satisfies $\iota(\xi_{_M})\beta=0$ for all $\xi \in \mathfrak{g}$. The collection of basic forms is a subcomplex $\Omega^*_{\textrm{bas}}(M) \subset \Omega^*(M)$ whose cohomology is isomorphic to the singular cohomology $H^*(M/G;\R)$ of the orbit space. Thus, the basic cohomology of a manifold is a topological invariant so that if $M$ and $N$ are homotopy equivalent, then $H^*_{\textrm{bas}}(M) = H^*_{\textrm{bas}}(N)$. Furthermore, the Poincar\'{e} Lemma holds for basic cohomology. If $p:M \times \R \to M$ is projection on the first factor, then the induced map $p^*:\Omega^*(M) \to \Omega^*(M \times \R)$ induces an isomorphism on basic cohomology. 

\begin{lemma} \label{basichomotopyproperty}
Suppose a compact Lie group, $G$, acts on a manifold, $M$. Let $i:Z  \hookrightarrow U$ be the inclusion of a submanifold into an invariant tubular neighborhood in $M$. If $\eta$ is a basic $(k+1)$-form on $U$ with $i^*\eta=0$, then there exists a basic $k$-form $\sigma$ with
\begin{enumerate}
\item $\sigma_{z}=0$ for all $z \in Z$, and
\item $d\sigma=\eta$ on $U$. 
\end{enumerate}
\end{lemma}
\begin{proof}
Define 
$$
\sigma =K\eta :=  \int_{0}^{1} \rho^*_{s}(\iota(\chi_{s})\nu) \ ds
$$
where 
\begin{eqnarray*}
\rho_s:U \times [0,1] &\to& U \\
(z,v,s) &\mapsto& (z,sv)
\end{eqnarray*}
 is equivariant and
$$
\chi_{s}(\rho_{s}) = \left.\frac{d}{dr}\right|_{r=s} \rho_{r}
$$
is invariant.
Note that the form $\sigma$ is invariant. When $s=0$ ({\em i.e.}, at points of $Z$), $\sigma=0$. Finally, since $\eta$ is basic, for any $\xi \in \mathfrak{g}$ we have 
\begin{eqnarray*}
\iota(\xi_{M})\sigma &=& \iota(\xi_{M})K\eta \\
&=& \iota(\xi_{M})\int_{0}^{1} \rho^*_{s}(\iota(\chi_{s})\eta) \ ds \\
&=& \int_{0}^{1} \iota(\xi_{M})\rho^*_{s}(\iota(\chi_{s})\eta) \ ds \\
&=& \int_{0}^{1} \rho^*_{s}(\iota(d(\rho_{s}^{-1}))(\xi_{M})\iota(\chi_{s})\eta) \ ds \\
&=& 0.
\end{eqnarray*}
We now prove that $d\sigma = \eta$. Let $p:U \to Z$ be the projection of $U$ onto $Z$ and note that since $\eta$ is closed and $i^*\eta=0$,
\begin{eqnarray*}
d\sigma &=& dK\eta+Kd\eta \\ 
&=& d\int_{0}^{1} \rho^*_{s}(\iota(\chi_{s})\eta) \ ds+\int_{0}^{1} \rho^*_{s}(\iota(\chi_{s})d\eta) \ ds \\
&=& \int_{0}^{1} \rho_{s}^*\mathcal{L}_{\chi_{s}}\eta \ ds\\
&=& \int_{0}^{1} \frac{d}{ds} \rho_{s}^*\eta \ ds\\
&=& \rho^*_{1}\eta-\rho^*_0\nu \\
&=& \eta-(i \circ p)^*(\eta)\\
&=& \eta.
\end{eqnarray*}\end{proof}

We are now ready to state a technical lemma that, along with a variant of Moser's Trick, we will use to prove the existence of an equivariant folded symplectomorphism given an equivariant diffeomorphism preserving moment maps. 

\begin{lemma}\label{basicform}
Let $(M,\omega, \Phi)$ and $(M', \omega', \Phi')$ be orientable, folded symplectic toric four-manifolds. Suppose that $g:M \to M'$ is an equivariant diffeomorphism such that $g^*\Phi' = \Phi$. Then, if $H^2(M/\T; \Z)=0$, there exists a basic one-form $\beta$ on $M$ such that
\begin{enumerate}
\item $d\beta = g^*\omega'-\omega$, and
\item for all $z \in Z$, $\beta_z(W_z)=0$ whenever $W_z \in \ker(\omega_s)_z$.
\end{enumerate}
\end{lemma}

\begin{proof}\label{techlemma}
The first fact to note is that $\eta = g^*\omega'-\omega$ is a basic form on $M$: it is closed, and it is invariant since $\omega,\omega'$ are invariant and $g$ is equivariant. Also, for any $\xi \in \mathfrak{t}$, 
\begin{eqnarray*}
\iota(\xi_{_M})\eta &=& \iota(\xi_{_M})g^*\omega'-\iota(\xi_{_M})\omega \\
&=& -d\Phi^{\xi}+d\Phi^{\xi} \\
&=& 0 
\end{eqnarray*}
since $g$ is moment-preserving. By assumption, $H^2(M/\T;\Z)=0$ and so $H^2(M/\T;\R)=0$. It follows that the basic cohomology $H^2_{\textrm{bas}}(M)=0$ as well. We conclude that there exists a basic one-form $\alpha$ on $M$ such that $d\alpha=\eta$. 

By Lemma \ref{lemmarestriction}, the restriction $i^*g^*\omega'-i^*\omega=0$. In light of Lemma \ref{basichomotopyproperty}, there exists a basic one-form $\sigma$ on an invariant tubular neighborhood $U_1$ of $Z$ in $M$ such that $\sigma_{z}=0$ for all $z \in Z$ and $d\sigma = \eta|_{U_1}$. Note that $U_1$ is orientable and is equivariantly diffeomorphic to $Z \times \R$.

Let $U_2=M\setminus{Z}$ and be $\{\rho_{U_1},\rho_{U_2}\}$ an invariant partition of unity subordinate to the open cover $\{U_1,U_2\}$. Define on $M$ the one-form
$$
\tau:=\rho_{U_1}\sigma+\rho_{U_2}\alpha.
$$
Note that $\tau$ is basic and $\tau_{z}=0$ for all $z \in Z$, but
$$
d\tau = \eta+d\rho_{U_1}\wedge{\sigma}+d\rho_{U_2}\wedge{\alpha}.
$$

The form $\sigma-\alpha$ is a closed, basic one-form on the tubular neighborhood $U_1$. By the Poincar\'{e} Lemma for basic cohomology, the map $p^*:H^1_{\textrm{bas}}(Z)\to H^1_{\textrm{bas}}(U_1)$ induced by the projection $p:U \to Z$ is an isomorphism and so $\sigma - \alpha = p^*\nu+df$ where $\nu \in H^1_{\textrm{bas}}(Z)$ and $f \in C^{\infty}(U_1)$. Rewriting, we have $\alpha = \sigma-p^*\nu-df$. 

Define
$$
\beta:= \tau -\rho_{U_1}(p^*\nu)+fd\rho_{U_1}.
$$
Then, $\beta$ is a basic one-form on $M$ and 
\begin{eqnarray*}
d\beta &=& d\tau-d\rho_{U_1}\wedge{p^*\nu}-d\rho_{U_1}\wedge{df} \\
&=& \eta + d\rho_{U_1}\wedge{\sigma}+d\rho_{U_2}\wedge{\alpha}+d\rho_{U_1}\wedge{p^*\nu}-d\rho_{U_1}\wedge{df} \\
&=& \eta +d\rho_{U_1}\wedge{\sigma}-d\rho_{U_1}\wedge{\alpha}+d\rho_{U_1}\wedge{p^*\nu}-d\rho_{U_1}\wedge{df} \\
&=& \eta.
\end{eqnarray*}
Since $\beta$ is basic, $\iota(V)\beta=0$ for all $V \in \ker(i^*\omega)=\ker(i^*\omega_{s})$ by Lemma \ref{Vtangenttoorbits}. Lemma \ref{kerneltalk} implies that the vector field $W^{s} \in \ker(\omega_s)$ tangent to $M$ differs from the radial vector field $W=\frac{\partial}{\partial{t}} \in \ker(\omega)$ by a vector field tangent to orbits,  Hence, 
\begin{eqnarray*}
\iota(W^{s}_{z})\beta_z &=& \iota\left(\frac{\partial}{\partial{t}}\right) \beta_{z} \\ 
&=& \iota\left(\frac{\partial}{\partial{t}}\right)(\tau_z+\rho_{U_1}(z)(p^*\nu)_z-f(z)(d\rho_{U_1})_z) \\
&=& \iota\left(\frac{\partial}{\partial{t}}\right)\rho_{U_1}(z)(p^*\nu)_z\\
&=& 0
\end{eqnarray*}
(recall that $p^*\nu$ is the pullback of a form on $Z$). Thus, $\ker(\beta) = \ker(\omega_s)$ at all points $z \in Z$. \end{proof}

We end this section with a theorem stating that, under a topological assumption on the orbit space, the existence of a moment-preserving equivariant diffeomorphism forces the existence of an isomorphism in the orientable, four-dimensional case.

\begin{theorem}\label{global}
Let $(M,\omega, \Phi)$ and $(M', \omega', \Phi')$ be compact, orientable, folded symplectic toric four-manifolds. Suppose that $H^2(M/\T; \Z)=0$. If there exists an equivariant diffeomorphism $g:M \to M'$ such that $g^*\Phi' = \Phi$, then $M$ and $M'$ are isomorphic. 
\end{theorem}

\begin{proof}
By Lemma \ref{basicform}, there exists a basic form $\beta$ such that $d\beta = g^*\omega'-\omega$. Since $\beta_{z}(W_z)=0$ for all $W_z \in \ker(\omega_{s})_z$, there exists $X_s$ such that $\iota(X_s)\omega_s = -\beta$. For all $\xi \in \mathfrak{t}$ we have
\begin{eqnarray*}
-\iota(X_s)d\Phi^{\xi} &=& \iota(X_s)\iota(\xi_{M})\omega_s \\
&=& \iota(\xi_{M})\beta \\
&=& 0 
\end{eqnarray*}
since $\beta$ is basic. Hence, the flow $\{\phi_s\}$ preserves the moment map $\Phi$. 

The pullbacks $\phi_s^*\omega_s$ are independent of $s$ since
\begin{eqnarray*}
\frac{d}{ds}(\phi_s^*\omega_s) &=& \phi_s^*\left(\frac{d}{ds}\omega_s+\mathfrak{L}_{X_s}\omega_s\right) \\
&=& \phi_s^*\left(g^*\omega'-\omega-d\beta\right) \\ &=& 0. 
\end{eqnarray*}
Consequently, $\omega = \phi_0^*\omega_0=\phi_1^*\omega_1=\phi_1^*g^*\omega'$. Define $\psi:M \to M'$ by $\psi:=g \circ \phi_1$. By construction, $\psi$ is an isomorphism. \end{proof}

\section{From Local to Global}

We conclude by combining the results of the previous sections to prove Theorem \ref{maintheorem}, which we restate for convenience.

\begin{theorem*}
Let $(M, \omega, \Phi)$ and $(M', \omega', \Phi')$ be orientable folded symplectic toric four-manifolds. Suppose that $H^2(M/\T; \Z)=0$. Then, $M$ and $M'$ are isomorphic if there exists a diffeomorphism $h:M/\T \to M'/\T$ of manifolds with corners such that $h^*\underline{\Phi}' = \underline{\Phi}$. 
\end{theorem*}

\begin{proof}
Since $h:M/\T \to M/\T$ is a diffeomorphism preserving orbital moment maps, Theorem \ref{localstructure} implies that $M$ and $M'$ are locally isomorphic. So, $h$ lifts locally to a collection of isomorphisms $\{g_{i}\}$ on an open cover $\mathcal{V}=\{V_i\}$ of $M$. The maps 
\begin{eqnarray*}
V_i \cap V_j &\to& V_i \cap V_j \\
p &\mapsto& (g_{i}^{-1}\circ{g_j})(p)
\end{eqnarray*}
are equivariant diffeomorphisms that fix orbits. By a theorem of Haefliger and Salem (Theorem 3.1, \cite{hs}), there exists a collection $\{a_{ij}\}$ of smooth maps $a_{ij}:V_i \cap V_j \to \T$ such that $(g_{i}^{-1}\circ{g_j})(p) = a_{ij}(p)\cdot{p}$. 

%We claim that $\{a_{ij}\}$ forms an element $[a] \in \check{H}^1(M/\T,\T)$, the \v{C}ech cohmology of the orbit space with respect to the sheaf of smooth functions with values in $\T$. 
Note that since the $a_{ij}$ preserve orbits, we may think of them as smooth, $\T$-valued functions on $M/\T$ and the collection $\{a_{ij}\}$ forms a class $[a] \in \check{H}^1(M/\T,\T)$, the \v{C}ech cohmology of the orbit space with respect to the sheaf of smooth functions with values in $\T$. 
%Let $\delta$ denote the \v{C}ech coboundary operator. On triple intersections $U_i \cap U_j \cap U_k$,
%\begin{eqnarray*}
%(\delta{a})_{ijk} &=& a_{ij} \circ a_{jk} \circ a_{ik}^{-1} \\
%&=& g_{i}^{-1}\circ{g_j}\circ{g_{j}^{-1}}\circ{g_k}\circ{g_{k}^{-1}}\circ{g_i} \\
%&=& \textrm{Id}_{ijk}.
%\end{eqnarray*}
%Hence, $a = \{a_{ij}\}$ is a $1$-cocycle.
Let $\ell$ and $\mathfrak{t}$ be the sheaves of smooth functions on $M/\T$ with values in the lattice $\ker\{\exp:\mathfrak{t} \to \T\}$ and the Lie algebra of $\T$, respectively. The short exact sequence 
$$0 \to \ell \to \mathfrak{t} \to \T \to 0$$
induces the long exact sequence in cohomology
$$
\cdots \to H^{k-1}(M/\T,\mathfrak{t})\to H^{k-1}(M/\T,\T) \to H^{k}(M/\T,\ell) \to H^{k}(M/\T,\mathfrak{t}) \to \cdots .
$$
Since $M/T$ is paracompact, it admits partitions of unity so the sheaf $\mathfrak{t}$ is fine and hence acyclic. Thus, $H^{k}(M/T,\mathfrak{t})=0$ for all $k>0$. As a consequence, $H^1(M/T,\T)$ is isomorphic to $H^2(M/T,\ell)$. By assumption, $H^2(M/T,\Z)=0$ implying $H^2(M/T,\ell)=0$ as well. The class $[a]$ is thus trivial in $H^1(M/T,\T)$ and the diffeomorphisms $g_i$ glue to give a global diffeomorphism $g:M \to M'$. Since $g$ is a lift of $h$ and $h$ preserves orbital moment maps, we have $g^*\Phi' = \Phi$. By Theorem \ref{global}, this implies that there exists an isomorphism $\psi:M \to M'$. 
\end{proof}

\bibliographystyle{amsalpha}
\bibliography{thesisbiblio}

\end{document}